\documentclass{amsart}
\usepackage{amsmath, amssymb, amsfonts, amscd, amsthm, mathrsfs}
\usepackage{tikz}
\usetikzlibrary{matrix,arrows}
\usepackage{comment}

\renewcommand{\hat}{\widehat}
\renewcommand{\tilde}{\widetilde}
\renewcommand{\hat}{\widehat}
\renewcommand{\tilde}{\widetilde}
\renewcommand{\bar}{\overline}

\newcommand{\pfloor}[1]{\lfloor #1 \rfloor}
\newcommand{\pceiling}[1]{\lceil #1 \rceil}

\newcommand{\A}{\mathbf{A}}

\newcommand{\F}{\mathbf{F}}
\newcommand{\Q}{\mathbf{Q}}
\newcommand{\Z}{\mathbf{Z}}
\newcommand{\R}{\mathbf{R}}

\newcommand{\cS}{\mathcal{S}}

\newcommand{\Tr}{\mathrm{Tr}}
\newcommand{\ord}{\mathrm{ord}}

\newcommand{\NP }{\mathrm{NP}}
\newcommand{\HP }{\mathrm{HP}}
\newcommand{\GNP }{\mathrm{GNP}}
\newcommand{\sgn}{\mathrm{sgn}}

\newcommand{\Zeta}{\mathrm{Zeta}}

\newcommand{\cont}{\mathrm{cont}}
\newcommand{\MaxPrime}{\mathrm{MaxPrime}}

\theoremstyle{plain}
\newtheorem{theorem}{Theorem}[section]
\newtheorem{proposition}[theorem]{Proposition}
\newtheorem{lemma}[theorem]{Lemma}

\newtheorem{corollary}[theorem]{Corollary}
\newtheorem{conjecture}[theorem]{Conjecture}
\theoremstyle{definition}

\newtheorem{remark}[theorem]{Remark}

\newtheorem{acknowledgments}{Acknowledgments}

\begin{document}

\title[Generic $A$-family of exponential sums]
{Generic $A$-family of exponential sums}

\subjclass[2000]{11,14}

\author{Hui June Zhu}
\address{
Department of mathematics,
State University of New York at Buffalo,
Buffalo, NY 14260}
\email{hjzhu@math.buffalo.edu}
\date{\today}

\begin{abstract}
Let $\vec{s}:=(s_1,s_2,\ldots,s_m)$ with 
$s_1<\cdots<s_m$ being positive integers.
Let $\A(\vec{s})$ be the space of all 1-variable polynomials 
$f(x) = \sum_{\ell=1}^{m}a_\ell x^{s_\ell}$ parametered by coefficients 
$\vec{a}=(a_1,\ldots,a_m)$ with $a_m\neq 0$.
We study the $p$-adic valuation of the roots of 
the $L$-function of exponential sum of $\bar{f}$ for mod $p$ reduction of any
closed point $f\in \A(\vec{s})(\bar\Q)$.
Let  $\NP(\bar{f})$ be the normalized $p$-adic Newton polygon of the $L$ function of exponential sums of $\bar{f}$. 
Let $\GNP(\A(\vec{s}),\bar\F_p)$ 
be the generic Newton polygon for $\A(\vec{s})$ over $\bar\F_p$,
and let $\HP(\A(\vec{s})):=\NP_p(\prod_{i=1}^{d-1}(1-p^{\frac{i}{d}}T))$
be the absolute lower bound of $\NP(\A(\vec{s}))$.
One knows that $\NP(\bar{f})\prec \GNP(\A(\vec{s});\bar\F_p)\prec\HP(\A(\vec{s}))$ for 
all prime $p$, and equalities hold when $p\equiv 1\bmod d$.
The equality does not generally hold for other residues.
In the case $\vec{s}=(s,d)$ with $s<d$ coprime we 
provide a computational method to determine 
$\GNP(\A(s,d),\bar\F_p)$ explicitly
by constructing its generating polynomial
$H_r\in\Q[X_{r,1},X_{r,2},\cdots, X_{r,d-1}]$ 
for each residue class $p\equiv r\bmod d$.
For $p\equiv r\bmod d$ large enough 
$H_r\neq 0$ has its lowest degree (nonzero) terms
$\sum_{n=1}^{d-1}h_{r,n,k_{r,n}}X_{r,n}^{k_{r,n}}$ if and only if
$\GNP(\A(s,d),\bar\F_p)$ has 
its breaking points after the origin   
$$(n, \frac{n(n+1)}{2d}+\frac{(1-\frac{s}{d})k_{r,n}}{p-1})$$
where $n=1,2,\cdots,d-1$.
If $a\neq 0$ then for any $f=x^d+ax^s\in \A(s,d)(\bar\Q)$ and 
for any prime $p\equiv r\bmod d$  large enough  
we have that $\NP(\bar{f}) = \GNP(\A(s,d);\bar\F_p)$ 
and 
$$
\lim_{\stackrel{p\equiv r\bmod d}{p\rightarrow\infty}}\NP(\bar{f}) = \HP(\A(s,d)).
$$

Our method applies to compute the generic Newton polygon of 
Artin-Schreier family 
$y^p-y=x^d+ax^s$ parametered by $a$ for $p$ large enough.
\end{abstract}

\maketitle


\section{Introduction}

Let $\vec{s}:=(s_1,s_2,\ldots,s_m)$ with 
$s_1<\cdots<s_m$ being positive integers.
Let $\A(\vec{s})$ be the space of all 1-variable polynomials 
$f(x) = \sum_{\ell=1}^{m}a_\ell x^{s_\ell}$ parametered by coefficients 
$\vec{a}=(a_1,\ldots,a_m)$ with $a_m\neq 0$. Without loss of generality 
we set $a_m=1$.
Fix a primitive $p$-th root of unity $\zeta_p$.
Let $f=\sum_{\ell=1}^{m}a_\ell x^{s_\ell} \in \A(\vec{s})(\bar\Q)$ be a closed point, that is, $\vec{a}\in\bar\Q^m$.
Let $\wp$ be a prime ideal in the number field $\Q(a_1,\ldots,a_m)$ 
lying over $p$, suppose its 
residue field is $\F_q$ for some $p$-power $q$. 
For any $k\in\Z_{\geq 1}$ 
let the $k$-th exponential sum of $\bar{f}:=f\bmod \wp$ in $\F_q[x]$ be 
$$
S_k(\bar{f})=\sum_{x\in\F_{q^k}}\zeta_p^{\Tr_{\F_q/\F_p}(\bar{f}(x))}
$$
and let the $L$ function of the exponential sum of $\bar{f}/\F_q$ to be
$$
L(\bar{f}/\F_q;T) = \exp\sum_{k=1}^{\infty} S_k(\bar{f}) T^k/k.
$$
It is known that 
$L(\bar{f}/\F_q;T) =\sum_{i=0}^{d-1}c_i T^i$
lies in $\Z[\zeta_p][T]$ with $c_0=1$.
The (normalized) $p$-adic Newton polygon of $L(\bar{f}/\F_q;T)$
denoted by $\NP(\bar{f}):=\NP_q(L(\bar{f}/\F_q;T))$.
that is, the lower convex hull of 
the points $(i,\ord_pc_i/(\log_p q))$ for $i=0,1,\ldots, d-1$
in the real plane $\R^2$. 
In other words, it is the $q$-adic 
Newton polygon as 
for any $c\in\bar\Q$ we write $\ord_q(c)=\ord_p(c)/(\log_pq)$.
Consider all Newton polygons with the same domain as piece-wise 
linear functions, we define an order $\NP_1\prec \NP_2$ if $\NP_1$
lies over $\NP_2$. For each prime $p$, 
there exists a lower bound for $\NP(\bar{f})$ by 
the Grothendieck-Katz specialization theorem (see \cite{Katz})
for all such Newton polygons, namely
$$\GNP(\A(\vec{s});\bar\F_p):=\inf_{\bar{f}\in \A(\vec{s})(\bar\F_p)}\NP(\bar{f})$$ exists.
The infimum is taken over all Newton polygons $\NP(\bar{f})$
as $\bar{f}\in \A(\vec{s})(\bar\F_p)$ with the partial order described above.

In this paper we shall always represent 
a Newton polygon 
by its breaking points coordinates in $\R^2$ after origin.
Let
$$
\HP(\A(\vec{s})):=\NP_p(\prod_{i=1}^{d-1}(1-p^{\frac{i}{d}}T)).
$$
In the literature $\HP(\A(\vec{s}))$ is often called the Hodge polygon
of $\A(\vec{s})$, and its breaking points after origin are
$(n,\frac{n(n+1)}{2d})$ for $n=1,\ldots,d-1$.
It is known that 
\begin{eqnarray}\label{E:Katz}
\NP(\bar{f})\prec
&\GNP(\A(\vec{s});\bar\F_p) &
\prec \HP(\A(\vec{s}))
\end{eqnarray}
and their endpoints coincide (see \cite{AS89}).
In fact this inequality holds for more general families of Laurent polynomials
in multivariables (see for instance \cite{AS89}). 
For $p\equiv 1\bmod d$ we have all three polygons coincide,
but it is not the case for other residue classes of the prime $p$.
In fact, $\GNP$ generally depends on not only the residue class of $p$ but 
also $p$ itself, and from experimental data for lower degree cases
one observes that $\GNP$ has a formula for certain residue families,
and we prove this in this paper and give explicit formulas.

For $\vec{s}=(1,2,\ldots,d)$,
Wan has conjectured that a generic polynomial of degree $d$
in $\A(\vec{s})(\bar\Q)$ has  
its Newton polygon at each mod $p$ reduction 
approaching to the absolute lower bound $\HP(\A(\vec{s}))$ as $p$ goes to infinity. 
This conjecture was proved in \cite{Zhu03} where it is also proved that 
Wan's conjecture applies to a 1-parameter family $\A(1,d)$.
In this paper we generalize a main theorem of \cite{Zhu03}
from $\A(1,d)$ to $\A(s,d)$. Our major contribution of the current paper is 
to provide an explicit method allowing one to compute 
$\GNP(\A(s,d),\bar\F_p)$ for every prime $p$ large enough, 
we are also developing this method for more general families in the future
paper. We prove in this paper the generic Newton polygon at each prime $p$
may be computed globally over $\Q$ instead, and for $p$ large enough it
has a formula depending only on the residue of $p\bmod d$.

For any $a$ in $\bar\Q$ we use $\MaxPrime(a)$ to denote 
the maximal prime factor of $N_{\Q(a)/\Q}(a)$ in $\Q$.
Let $\MaxPrime(a_1,a_2,\ldots)$ be the maximum of these
$\MaxPrime(a_i)$'s.

For any $2\leq r\leq d-1$ coprime to $d$, 
we construct a generating polynomial
$H_r\in\Q[X_{r,1},\ldots,X_{r,d-1}]$ 
for $\GNP(\A(s,d),\bar\F_p)$ in Section 2, see (\ref{E:H}).
Key result of this paper lies in the following theorem:

\begin{theorem}\label{T:main}
Let $s<d$ be coprime positive integers.
Suppose $H_r\neq 0$ with lowest degree (nonzero) terms
$\sum_{n=1}^{d-1}h_{r,n,k_{r,n}}X_{r,n}^{k_{r,n}}$.
Let $N_{s,d,r}:=\max(s(d-1),d+\max_n(k_{r,n}), 2(d-s)\max_n(k_{r,n}),\MaxPrime_n(h_{r,n,k_{r,n}}))$.
Then for every prime $p\equiv r \bmod d$ and $p>N_{s,d,r}$, 
we have $\GNP(\A(s,d),\bar\F_p)$ with breaking points after the origin at
$(n,\frac{n(n+1)}{2d}+\frac{(1-\frac{s}{d})k_{r,n}}{p-1})$ 
for $n=1,\ldots,d-1$.
Conversely, suppose
$\GNP(\A(s,d),\bar\F_p)$ has its breaking points after the origin 
$(n,\frac{n(n+1)}{2d}+\frac{(1-\frac{s}{d})k_{r,n}}{p-1})$ for all $n$
for all prime $p>\max(s(d-1),d+\max(k_{r,n}),2(d-s)\max_n(k_{r,n}))$
then 
$H_r\neq 0$ with lowest degree (nonzero) terms
$\sum_{n=1}^{d-1}h_{r,n,k_{r,n}}X_{r,n}^{k_{r,n}}$.

Let $f=x^d+ax^s\in\A(s,d)(\bar\Q)$ and 
we write $\bar{f}$ its reduction mod a prime in $\Q(a)$ over $p$.
Suppose $a\neq 0$.
Then for all prime $p\equiv r\bmod d$ and 
$p > \max(N_{s,d,r},\MaxPrime(a))$ we have 
\begin{eqnarray*}
\NP(\bar{f}) &=&\GNP(\A(s,d);\bar\F_p)\\
\lim_{\stackrel{p\equiv r\bmod d}{p\rightarrow\infty}}
\NP(\bar{f}) &=&\HP(\A(s,d)).
\end{eqnarray*}
\end{theorem}

If $s<d$ are not coprime, then the statements relating to $\HP(\A(s,d))$
in Theorem \ref{T:main} are false. However, there exists 
$\GNP(\A(s,d),\bar\F_p)$ in that case and the situation was carried out in 
\cite{BFZ08}.

As a byproduct we show that 
the generic Newton polygon $\GNP(\A(s,d),\bar\F_p)$ for $p\equiv r\bmod d$ 
and $p$ is large enough has a formula. We shall also see that 
each of these generic Newton polygons can be achieved over 
$\F_p$.

For our family $\A(s,d)$ we construct a semi-linear Fredholm $A$-matrix $M'$
which represents Dwork's Frobenius matrix over $\F_p$.
The $L$ function of a closed special  point $\bar{f}\in\A(s,d)(\F_q)$
with $q=p^c$ is determined by the Fredholm $A$-matrix
$M'_c:=M'\cdot (M')^{-\tau}\cdots \cdot (M')^{-\tau^{c-1}}$
where $\tau$ is the Frobenius map.
However, this infinite matrix is notoriously messy to 
compute if one ever can, and furthermore $c$ can be arbitrarily large
and this changes the corresponding $L$-function fundamentally.
Meanwhile, the Fredholm determinant of $M'_c$ also depends on 
the prime $p$ intricately. Our method here is: 
we first work out complete solution set to the Frobenius problem
in $2$-dimensional case 
(it is not yet known one can explicitly compute all 
such complete solution set for $>2$-dimensional cases).
Then for $p$ large enough we approximate 
our Fredholm $A$-matrix by a finite one.
This finite Fredholm $A$-matrix can be explicitly written down,
and most remarkably its $p$-adic order has a formula 
for each residue of $p\bmod d$.
We prove in this paper that the generic $A$-families over $\bar\F_p$ 
for $p$ large enough are all the imagines of a 
global generic object over $\bar\Q$.

Our theorem has application to Artin-Schreier families.
For any $f=x^d+ax^s\in\A(s,d)(\bar\Q)$ 
let $X_f: y^p-y= f(x)$ mod $\wp$
be the corresponding mod $p$ reduction over 
some finite field $\F_q$. 
It is known 
that the Zeta function $\Zeta(X_f/\F_q;T)$ of $X_f/\F_q$ in variable $T$
lies in $\Q[T]$ and its numerator (as the core factor)
is a polynomial of degree $(d-1)(p-1)$.
In fact it is known that 
$$
\Zeta(X_f/\F_q;T) = \frac{N_{\Q(\zeta_p)/\Q}(L(\bar{f}/\F_q;T))}{(1-T)(1-qT)}
$$
where the norm being defined
as the product of all Galois conjugates
of the polynomial $L(\bar{f}/\F_q;T)\in\Q[\zeta_p][T]$
where the automorphism of $\Q(\zeta_p)/\Q$ 
acts trivially on the variable $T$.
Let the Newton polygon 
$\NP(X_f/\F_q)$
of $X_f/\F_q$ 
be the $q$-adic Newton polygon of the numerator of $\Zeta(X_f/\F_q;T)$.
Thus $\NP(\bar{f}/\F_q)$ is precisely equal to 
$\NP(X_f/\F_q)$ shunk by a factor of $p-1$ horizontally and vertically,
that we denote this by 
$
\NP(\bar{f}/\F_q) = \NP(X_f/\F_q)/(p-1).
$
Then the following geometric application 
is an immediate corollary of Theorem \ref{T:main}.

\begin{corollary}\label{C:main}
Let $H_r\in\Q[X_{r,1},\ldots,X_{r,d-1}]$ be the generating polynomial 
constructed in Theorem \ref{T:main}.
Suppose 
$H_r=\sum_{n=1}^{d-1}h_{r,n,k_{r,n}}X_{r,n}^{k_{r,n}}+(\mbox{higher terms})$
for every $2\leq r\leq d-1$.
If $a\neq 0$ then for any 
$f=x^d+ax^s$ in $\A(s,d)(\bar\Q)$ 
and for any prime $p$  large enough we have 
$\frac{\NP(X_f/\F_q)}{p-1} = \GNP(\A(s,d),\bar\F_p)$ 
whose breaking points after origin are
$(n,\frac{n(n+1)}{2d}+\frac{(1-\frac{s}{d})k_{r,n}}{p-1})$ for $n=1,\ldots,d-1$,
and 
$$
\lim_{p\rightarrow \infty}\frac{\NP(X_f/\F_q)}{p-1} = \HP(\A(s,d)).
$$
\end{corollary}

This paper is organized as follows.
We first have some preliminary preparation in Section 2  and 
define generating polynomials $H_r$ for every $2\leq r\leq d-1$.
These polynomials in $\Q[X_{r,1},\ldots,X_{r,d-1}]$ 
depend only on $s,d$ and $r$ essentially. 
In fact, the most technical procedure in this 
paper is the construction of these global ($p$-free!) generating polynomials
that are linked to $p$-adic Fredholm determinant of the Frobenius
for all primes $p$ large enough.
Section 3 provides the bridges between these global polynomials $H_r$ 
over $\Q$ 
and the $p$-adic local analysis, especially under the condition that 
$p$ is large enough.
Section 4 develops Dwork theory for our 1-parameter 
$a$-family $\A(s,d)(\bar\F_p)$ for $p$ large enough.
We prove our main result Theorem \ref{T:main} in Section 4.

\begin{acknowledgments}
Research in this paper was partially supported by NSA 
mathematical science research grant 1094132-1-57192.
\end{acknowledgments}

\section{Frobenius problem and generating polynomials for $\GNP$}

\subsection{Preliminaries}

In this section we develop combinatorial and number theoretic 
preparations for our main theorem. These two lemmas are elementary
yet essential in the arguments of this paper.

\begin{lemma}\label{L:basic2}
Let $1\leq r<d$ be two coprime positive integers.
Let $h(z)$ be a fixed nonzero polynomial in $\bar\Q[z]$.
Then $h(-\frac{r}{d})\neq 0$ if and only if 
for all large enough prime 
$p\equiv r\bmod d$  we have $h(\pfloor{\frac{p}{d}})\in \bar\Z_p^*$.
\end{lemma}

\begin{proof}
For all prime $p$ large enough we have
$h(z)\in\bar\Z_p[z]$ obviously.
For such $p$ notice that $p\nmid d$, so we have
$\pfloor{\frac{p}{d}}\equiv -\frac{r}{d}\bmod p$ and hence 
$h(\pfloor{\frac{p}{d}})\in\bar\Z_p^*$ if and only if $h(-\frac{r}{d})\in\bar\Z_p^*$.

If $\theta:=h(-\frac{r}{d})\in\bar\Q^*$
then it is clear that $h(-\frac{r}{d})\in\bar\Z_p^*$ for all $p>\MaxPrime(\theta)$.
That is $h(\pfloor{\frac{p}{d}}) \in \bar\Z_p^*$ for all such $p$. 
The converse is clear.
\end{proof}

When $h(z)$ lies in $\Z[z]$ we have the following lemma
that yields an effective bound for $p$.
For any $h\in\bar\Q[z]$ let $h^o:=h/\cont(h)$ where $\cont(h)$ is the content of the polynomial $h$.

\begin{lemma}\label{L:basic}
Let $1\leq r\leq d-1$ for integers $d\in\Z_{>1}$.
Let $h(z)\in\Z[z]$. 

(1) If $h(-\frac{r}{d})\neq 0$ then 
$dz_0+r\nmid h(z_0)$ for all integers $z_0 \geq d^{\deg(h)-1}|h(-\frac{r}{d})|$.

(2) Suppose prime $p\nmid \cont(h)$ and $p>d$.
If $h(-\frac{r}{d})\neq 0$ then  
$h(\pfloor{\frac{p}{d}})\in\Z_p^*$ for all 
$p>\MaxPrime(h^o(-\frac{r}{d}))$;
conversely, if $h(\pfloor{\frac{p}{d}})\in\Z_p^*$ 
for any prime $p$, then $h(-\frac{r}{d})\neq 0$.
\end{lemma}
\begin{proof}
(1) Without loss of generality we assume $h(z)$ has its leading coefficient
$>0$. Taking long division algorithm in $\Q[z]$ we have
$h(z)=(dz+r)g(z)+R$ for unique 
$R=h(-\frac{r}{d})\in\Q$ and unique $g(z)\in\Q[z]$ with leading coefficient $>0$.
Suppose for $z_0\in\Z_{>0}$ we have $h(z_0)=(dz_0+r)C$ for some nonzero integer $C$ depending on $z_0$ of course. 
Then we have 
\begin{eqnarray*}
h(-\frac{r}{d}) &=& (dz_0+r)(C-g(z_0)).
\end{eqnarray*}
Let $h(z)=\sum_{i=0}^{m}h_iz^i$ for $h_i\in\Z$ and write $g(z)=\sum_{i=0}^{m-1}
g_iz^i$, then we have $g_{m-1}=h_m/d$ and $g_{i-1}=(h_i-r g_i)/d$ for all $i$.
Hence we have $d^m g_i\in \Z$ for all $i$.
Rewrite the above equation  below
\begin{eqnarray*}
d^mh(-\frac{r}{d}) &=& (dz_0+r)(d^mC-d^mg(z_0)).
\end{eqnarray*}
Since the left-hand-side is a fixed integer, and the factor
$d^m C-d^mg(z_0)$ is also an integer, 
we have that 
$d z_0 +r\leq d^m|h(-\frac{r}{d})|$. 
This says that 
if $dz_0+r > d^m|h(-\frac{r}{d})|$ or equivalently 
$z_0\geq d^{m-1} |h(-\frac{r}{d})|$,
then we have $dz_0+r \nmid h(z_0)$.

(2) 
Write $c=\cont(h)$.
Assume that prime $p=dz_0+r$ 
is coprime to $c$ with $z_0:=\pfloor{p/d}$.
Then $C=cC^o$ for some $C^o\in\Z$. 
Write $g=cg^o$ for $g^o\in \Q[x]$, we have
$h^o(z)=(dz+r)g^o(z)+h^o(-\frac{r}{d}) \in\Z[z]$
implies that $d^mg^o(z)\in\Z[z]$.
We have 
\begin{eqnarray*}
d^mh^o(-\frac{r}{d}) &=& p (d^mC^o-d^mg^o(z_0)).
\end{eqnarray*}
Since the left-hand-side is a fixed integer, and the factor
$d^m C^o-d^mg^o(z_0)$ is also an integer, 
we have that 
$p\leq \MaxPrime(h^o(-\frac{r}{d}))$.
This says that 
if prime $p >\MaxPrime(h^o(-\frac{r}{d}))$
then we have $p \nmid h^o(\pfloor{\frac{p}{d}})$, i.e.
$p\nmid h(\pfloor{\frac{p}{d}})$. 
The converse is clear since 
$\pfloor{\frac{p}{d}}\equiv -\frac{r}{d}\bmod p$ 
for $p>d$.
\end{proof}

Below we shall study solutions 
to the Frobenius problem with given two coprime integers.
We shall fix two coprime positive integers $d,s$ with $d>s$.
Every pair $(m,n)$ with $dn+sm=v$ is called a solution to 
the Frobenius problem of $(s,d)$ in this paper.
For any nonnegative integers $v>ds-d-s$ let 
$\beta_v(d,s):=\min(m+n)$ where the minimum is taking over 
all nonnegative integers $m,n$ such that $dn+sm=v$.
Such minimum $\beta_v(d,s)$ exists and is achieved uniquely at 
$m=(s^{-1}v\bmod d)$ and 
$n=\frac{v}{d}-\frac{sm}{d}$.
The following lemma should be known in the literature but we 
provide its statement and proof here for the paper to be self-contained.

\begin{lemma}\label{L:Frobenius}
Let $v=pi-j$ with $1\leq i,j\leq d-1$
and let $v>ds-d-s+1$ (or $p>s(d-1)$). Let $r=(p\bmod d)$.
\begin{enumerate}
\item Then the minimum is achieved uniquely
$\beta_{pi-j}(s,d) =m_{ij}+ n_{ij}$ at
\begin{eqnarray*}
m_{ij} &=& (s^{-1}(ri-j)\bmod d)\\
n_{ij} &=& \frac{pi-j}{d} -\frac{sm_{ij}}{d}=\pfloor{\frac{p}{d}}i+\frac{ri-j-sm_{ij}}{d}.
\end{eqnarray*}
\item We have
$\beta_{pi-j} (s,d)= \frac{pi-j}{d}+(1-\frac{s}{d})m_{ij}
\geq \pceiling{\frac{pi-j}{d}}
\geq \pfloor{\frac{pi}{d}}.
$ 
\item 
We have
$0\leq m_{ij}\leq d-1$
and 
$
\pfloor{\frac{pi}{d}} - s+1\leq n_{ij}\leq \pfloor{\frac{pi}{d}}.
$
\item A general solution to this Frobenius problem is 
$$
n_{ij}^\ell:=n_{ij}-s\ell,\quad m_{ij}^\ell:=m_{ij}+d\ell
$$
for some $0\leq \ell\leq \pfloor{\frac{pi-j}{ds}}$.
(The minimum $\beta$ is achieved if and only if $\ell=0$.)
The sum of these solutions is 
$$
m_{ij}^\ell+n_{ij}^\ell = \beta_{pi-j}(s,d) +(d-s)\ell.
$$
\end{enumerate}
\end{lemma}
\begin{proof}
We prove our statements for general integer $v>0$ 
first as the specialization to $v=pi-j$ does not alter the argument.
It follows from that $d>s$ that 
this minimum of $m_v+n_v$ is uniquely achieved when $m_v$ is minimal.
Let $m_v:= (s^{-1}v\bmod d)$ be the least nonnegative residue mod $d$.
It is clear that $m_v$ is the minimal nonnegative solution possible to the equation
$dn_v+sm_v=v$.
Let $n_v:=(v-sm_v)/d$. Since $v>ds-d-s+1$, we have 
$v>(d-1)(s-1)\geq m_v(s-1)$. Thus $v-s m_v > -m_v \geq -(d-1)$.
Since $n_v$ is an integer with $n_v>-(d-1)/d$ and hence $n_v\geq 0$. 
Therefore, $m_v,n_v$ are nonnegative integers satisfying the equation
$dn_v+sm_v=v$. 
The rest of the statements follow from the definition.
\end{proof}

Observe from Lemma \ref{L:Frobenius} that 
matrix $(m_{ij})$ is bounded in each entry by $d-1$, and 
it varies and exhausts the residue class on each row and each column.
Its value depends on $r=(p\bmod d)$.
On the other hand, 
each $n_{ij}$ lies in the small neighborhood of $\frac{pi}{d}$,
and hence it increases as $p$ increases, but each $n_{ij}^\ell<p$ for 
all $1\leq i,j\leq d-1$.

\subsection{Generating polynomials for $\GNP$}

The goal of this subsection is to define
the generating polynomials $H_r$ in $\Q[X_{r,1},\ldots,X_{r,d-1}]$ for every 
residue $2\leq r\leq d-1$ for given $s,d,r$.
This subsection is a dry run.
The readers who seek motivation should read Section 4 first.

The case for $r=1$ is known hence we will omit it entirely, in fact 
one can also write $H_1 =1$ for completeness.
The idea is that the generic $A$-determinant in the focus of our 
study depends only on the residue $r=(p\bmod d)$, not on $p$ itself. 
There is a generating polynomial for the generic $A$-family 
whose lowest degree terms encode the information of $\GNP(\A(s,d),\bar\F_p)$.


From now we fix $r,s$ with 
$2\leq r\leq d-1$ is coprime to $d$ and $1\leq n\leq d-1$.
For each $1\leq i\leq d-1$ we define a linear function in variable $z$
\begin{eqnarray}\label{E:n}
\tilde{n}_{r,i,+}(z) &:= & iz+\pfloor{\frac{ri}{d}}.
\end{eqnarray}

For any positive integer $t$ we denote the $t$-th {\em falling factorial power} 
of $Y$ by $[Y]_t:=Y(Y-1)\cdots(Y-t+1)$, where $Y$ lies in any ring containing $\Z$.
Below our $Y$ is either a rational number or a rational function in $\Q[z]$.
For $1\leq i,j\leq d-1$ 
recall $m_{ij}=(s^{-1}(ri-j)\bmod d)$ from Lemma \ref{L:Frobenius}
and let 
\begin{eqnarray}\label{E:t}
t_{ij} &:=& \pfloor{\frac{ri}{d}}-\frac{ri-j-sm_{ij}}{d}+s\ell_{ij}.
\end{eqnarray}
Let 
\begin{eqnarray}\label{E:k^o}
k_{r,n}^o&:=&\min_{\sigma\in S_n}\sum_{i=1}^{n} m_{i,\sigma(i)}.
\end{eqnarray}
For any $k\geq k_{r,n}^o$ let $\cS(k)$ be the set of all 
$\sigma\in S_n$ and $\ell_{i,\sigma(i)}\in\Z_{\geq 0}$
such that 
$$
\sum_{i=1}^{n}m_{i,\sigma(i)} +d\sum_{i=1}^{n}\ell_{i,\sigma(i)}=k.
$$
Suppose $\cS(k)$ is not empty, then for each  $(\sigma,\ell_{i,\sigma(i)})$ lies in 
$\cS(k)$ we define
\begin{eqnarray*}
\Theta_n &:=& \prod_{i=1}^{n}
   \frac{(d-1+(k-k_{r,n}^o))!}{(m_{i,\sigma(i)}+d\ell_{i,\sigma(i)})! }.
\end{eqnarray*}
Then we define a polynomial in variable $z$: 
\begin{eqnarray}\label{E:h_n}
\tilde{h}_{r,n,k}(z) &:=&
\sum_{(\sigma,\ell_{i,\sigma(i)})\in \cS(k)}\sgn(\sigma)\Theta_n
\prod_{i=1}^{n}[\tilde{n}_{r,i,+}(z)]_{t_{i,\sigma(i)}}.
\end{eqnarray}
If $\cS(k)$ is empty, define $\tilde{h}_{r,n,k}(z):=0$.

We remark that 
in practise it is not necessary to compute $k_{r,n}^o$
as one can replace $\Theta_n$ by 
$$
\Theta'_n:=\prod_{i=1}^{n} \frac{(d-1+k) !}{(m_{i,\sigma(i)}+d\ell_{i,\sigma(i)})!},
$$
and 
define the $\tilde{h}'_{r,n,k}(z)$ as in (\ref{E:h_n}) accordingly.
The following proposition 
shows that this replacement only change the function upto a constant factor.
Its proof follows immediately from the very definition
in (\ref{E:h_n}) and hence we omit.

\begin{proposition}\label{P:b}
Let notation be as above. 
Then 
$$\tilde{h}'_{r,n,k}(z) = \left(\frac{(d-1+k)!}{(d-1+k-k_{r,n}^o)!}\right)^n\tilde{h}_{r,n,k'}(z).$$
\end{proposition}

Define
\begin{eqnarray}\label{E:h^o}
\tilde{h}_{r,n,k}^o(z) &:=& \tilde{h}_{r,n,k}(z)/\cont(\tilde{h}_{r,n,k}(z)).
\end{eqnarray}
By Proposition \ref{P:b} this function is well defined, independent of the choice
of $\Theta_n$.

\begin{lemma}\label{L:bound}
Let notation be as above.  
Fix $2\leq r\leq d-1$ coprime to $d$ and $1\leq n\leq d-1$.
\begin{enumerate}
\item 
If $(\sigma,\ell_{i,\sigma(i)})$ lies in $\cS(k)$ with $b:=k-k_{r,n}^o\geq 0$,
then 
$\ell_{i,\sigma(i)}\leq \ell_b^o:=\pfloor{b/d}$.
\item Then $0\le t_{ij}\leq s(b+1)$ is an integer for all $1\leq i,j\leq d-1$
\item 
We have $\tilde{h}_{r,n,k}(z)\in \Z[z]$. 
Furthermore, $\tilde{h}_{r,n,k}^o(z)$ depends only on $d,s,r,b$
and $\deg(\tilde{h}_{r,n,k}^o(z)) \leq ns(b+1)$.
\end{enumerate}
\end{lemma}
\begin{proof} 
(1). Since $k=\sum_{i=1}^{n}m_{i,\sigma(i)}+d\sum_{i=1}^{n}\ell_{i,\sigma(i)}$
we have 
$d\sum_{i=1}^{n}\ell_{i,\sigma(i)}<b$.
Hence $\ell_{i,\sigma(i)}\leq \ell_b^o$.

(2).
Combining the result in Part (1) 
it remains to show that 
$t_{ij}^o:= \pfloor{\frac{ri}{d}}-\frac{ri-j-sm_{ij}}{d}$ satisfies
that $0\leq t_{ij}^o\leq s$. 
Since $m_{ij}=(s^{-1}(ri-j)\bmod d)$ and $\gcd(s,d)=1$
we have 
$sm_{ij}\equiv ri-j\bmod d$.
Hence $\frac{sm_{ij}-ri+j}{d}\in\Z$ and so $t_{ij}^o\in\Z$.
By Lemma \ref{L:Frobenius} we have that 
$$
t_{ij}^o \leq \frac{j+sm_{ij}}{d}
\leq \frac{(d-1)+s(d-1)}{d}\leq 
\frac{(s+1)(d-1)}{d}$$
and hence 
$t_{ij}^o\leq s$ since $t_{ij}^o\in\Z$.
Notice that 
$$\pfloor{\frac{ri}{d}}d\geq \pfloor{\frac{ri-j}{d}}d=ri-j-(ri-j\bmod d)
\geq ri-j-sm_{ij}.$$
This proves that 
$\pfloor{\frac{ri}{d}} \geq \frac{ri-j-sm_{ij}}{d}$.
That is, $t_{ij}^o\geq 0$.

(3). 
Let $b:=k-k_{r,n}^o$.
Write 
$\delta_{i,\sigma(i)}:= 
(d-1+b)-(m_{i,\sigma(i)}+d\ell_{i,\sigma(i)})$,
then 
$$\delta_{i,\sigma(i)} = 
(d-1-m_{i,\sigma(i)})+(b-d\ell_{i,\sigma(i)}) \geq 0$$
by Lemma \ref{L:Frobenius}  and part (1) above.
Hence $\Theta_n$ 
is just the product of $\delta_{i,\sigma(i)}$-th falling factorial power of $(d-1+b)$
\begin{eqnarray*}
\Theta_n &=&\prod_{i=1}^{n}\frac{(d-1+b)!}{(m_{i,\sigma(i)}+d\ell_{i,\sigma(i)})!}\\
&=&\prod_{i=1}^{n}[d-1+b]_{\delta_{i,\sigma(i)}}\\
&=& \prod_{i=1}^{n}(d+b-1)(d+b-2)\cdots (d+b-1-\delta_{i,\sigma(i)}).
\end{eqnarray*}
It is clear that this is an integer depending only on
$d,s,r$ and $b$. 

On the other hand, 
the $t_{i,\sigma(i)}$-th falling factoring power 
of $\tilde{n}_{r,i,+}(z)$ is 
\begin{eqnarray*}
[\tilde{n}_{r,i,+}(z)]_{t_{i,\sigma(i)}}
&=& \tilde{n}_{r,i,+}(z)(\tilde{n}_{r,i,+}(z)-1)\cdots (\tilde{n}_{r,i,+}(z)-t_{i,\sigma(i)}+1)\\
&=& (iz+\pfloor{\frac{ri}{d}})
(iz+\pfloor{\frac{ri}{d}}-1)
\cdots (iz+\pfloor{\frac{ri}{d}}-
t_{i,\sigma(i)}+1).
\end{eqnarray*}
It lies in $\Z[z]$ of degree $t_{i,\sigma(i)}\leq s(b+1)$ (by Part (2)), and coefficients
is determined by $d,s,r,b$.
Thus by Proposition \ref{P:b}, 
$\tilde{h}^o_{r,n,k}(z)\in\Z[z]$ 
is of degree $\leq \max_{\sigma}\sum_{i=1}^{n}t_{i,\sigma(i)}$ and 
hence $\leq ns(b+1)$. 
\end{proof}

Fix $2\leq r\leq d-1$ and $1\leq n\leq d-1$.
Let $k$ range over integers $\geq k_{r,n}^o$ and compute
$
h_{r,n,k}:=\tilde{h}_{r,n,k}^o(-\frac{r}{d})
$
until $h_{r,n,k}\neq 0$ for some $k$ (if exists). 
Let $X$ be a variable, let 
\begin{eqnarray}\label{E:H_n}
H_{r,n}(X) &:=& \sum_{k_{r,n}^o\leq k< k_{r,n}^o+\infty}
h_{r,n,k} X^k,
\end{eqnarray}
Let 
\begin{eqnarray}\label{E:H}
H_r&:= & \sum_{n=1}^{d-1}H_{r,n}(X_{r,n}).
\end{eqnarray}
Notice that these two generating polynomials with coefficient in $\Q$, where
$H_r$ depends only on $d,s,r$ and $k$.

We are able to explicitly construct nonzero $H_r$ for 
all $d\leq 5$, namely the following conjecture is verified for 
all $d\leq 5$.

\begin{conjecture}\label{Conj:1}
Let $H_r$ be as defined in (\ref{E:H}) above.
For every $2\leq r\leq d-1$ coprime to $d$ 
we have $H_r\neq 0$ and its lowest degree 
term is of the form $\sum_{n=1}^{d-1}h_{r,n,k} X_{r,n}^k$ ($h_{r,n,k}\neq 0$)
for some bounded $k\in\Z_{\geq 0}$.
\end{conjecture}

\begin{lemma}\label{L:kappa} 
Fix $1\leq n\le d-1$ and $2\leq r\leq d-1$ coprime to $d$.
Let $b=k-k_{r,n}^o\ge 0$.
\begin{enumerate}
\item
Define 
\begin{eqnarray*}
\kappa_{r,n,k} := ((d-1+b)!)^n\prod_{i=1}^{n} \pfloor{\frac{pi}{d}}!.
\end{eqnarray*}
Then we have
$\kappa_{r,n,k}\in \Z$; and 
$\kappa_{r,n,k}\in (\Z\cap\Z_p^*)$ 
for prime $p \ge d+b$.
\item
Define 
\begin{eqnarray*}
\alpha_{r,n,k} &:=& 
\sum_{(\sigma,\ell_{i,\sigma(i)})\in\cS(k)}
\sgn(\sigma) 
\prod_{i=1}^{n}\cfrac{1}{m_{i,\sigma(i)}^{\ell_{i,\sigma(i)}}!n_{i,\sigma(i)}^{\ell_{i,\sigma(i)}}!}.
\end{eqnarray*}
We have $\alpha_{r,n,k}\in \Q\cap\Z_p$ for all prime $p\equiv r\bmod d$ 
and $p\geq d+b$.
Furthermore, we have
$$\tilde{h}_{r,n,k}(\pfloor{\frac{p}{d}})
=\kappa_{r,n,k}\cdot \alpha_{r,n,k}.
$$
\end{enumerate}
\end{lemma}
\begin{proof}
(1) 
It is clear that $\kappa_{r,n,k}\in\Z$.
Since $n\leq d-1$ we have $\pfloor{pi/d}\leq p-1$
and hence $\pfloor{\frac{pi}{d}} \in\Z_p^*$ for all $p$.
On the other hand, $d-1+b<p$ by our hypothesis and 
hence $(d-1+b)!\in\Z_p^*$ too.
Hence $\kappa_{r,n,k}\in\Z_p^*$ for $p\geq b+d$.

(2) 
We first observe that 
$\tilde{n}_{r,i,+}(\pfloor{\frac{p}{d}})=i\pfloor{\frac{p}{d}}+\pfloor{\frac{ri}{d}}
=\pfloor{\frac{pi}{d}}$ (by writing $p=\pfloor{\frac{p}{d}}d+r$).
Secondly we notice that for all $i,j$
$$t_{ij}=\tilde{n}_{r,i,+}(\pfloor{\frac{p}{d}})-n_{ij}^{\ell_{ij}}= \pfloor{\frac{pi}{d}}-n_{ij}^{\ell_{ij}}.$$ 
Thus we have
$$
[\tilde{n}_{r,i,+}(\pfloor{\frac{p}{d}})]_{t_{i,\sigma(i)}}
=\frac{\tilde{n}_{r,i,+}(\pfloor{p/d})!}{n_{i,\sigma(i)}^{\ell_{i,\sigma(i)}}!}
=\frac{\pfloor{pi/d}!}{n_{i,\sigma(i)}^{\ell_{i,\sigma(i)}}!}.
$$
Therefore
\begin{eqnarray*}
\tilde{h}_{r,n,k}(\pfloor{\frac{p}{d}})
&=&\kappa_{r,n,k}
\sum_{(\sigma,\ell_{i,\sigma(i)})\in\cS(k)} 
\sgn(\sigma)\prod_{i=1}^{n}
\frac{1}{m_{i,\sigma(i)}^{\ell_{i,\sigma(i)}}!
n_{i\sigma(i)}^{\ell_{i,\sigma(i)}}!}
\end{eqnarray*}
which proves our statement.

Since $n_{ij}^\ell<p$ by Lemma \ref{L:Frobenius} 
and hence $n_{ij}^\ell!$ is in $\Z_p^*$. 
By Lemma \ref{L:bound}
we have $\ell_{i,\sigma(i)}\leq \pfloor{\frac{b}{d}}$ and 
hence $m_{ij}^\ell=m_{ij}+d\ell\leq d-1+d\pfloor{\frac{b}{d}}<p$ 
for $p\geq d+b$
and it follows $m_{ij}^\ell !  \in\Z_p^*$.
Thus  $\alpha_{r,n,k}\in\Z_p$ for $p\geq d+b$.
\end{proof}

\begin{proposition}\label{P:A^k}
Fix $r,n$ as above. 
\begin{enumerate}
\item Then 
$H_{r,n}=h_{r,n,k}X^k+(\mbox{higher terms})$ if and only if 
$k\geq 0$ is the least such that $\tilde{h}_{r,n,k}^o(-\frac{r}{d})\neq 0$.

\item If 
$\tilde{h}_{r,n,k}^o(\pfloor{\frac{p}{d}}) \in\Z_p^*$
for all prime $p\equiv r\bmod d$ and $p>\max(d,\MaxPrime(h_{r,n,k}))$
then $\tilde{h}_{r,n,k}^o(-\frac{r}{d})\neq 0$.
Conversely
if $\tilde{h}_{r,n,k}^o(-\frac{r}{d})\neq 0$ for $p\equiv r\bmod d$ and $p>d$ then 
$\tilde{h}_{r,n,k}^o(\pfloor{\frac{p}{d}}) \in\Z_p^*$.

\item 
For any prime $p\equiv r\bmod d$ and $p>d+k$ we have 
$\tilde{h}_{r,n,k}^o(\pfloor{\frac{p}{d}}) \in\Z_p^*$
if and only if 
$\alpha_{r,n,k}\in\Z_p^*$.

\item
If $H_{r,n}=h_{r,n,k}X^k+(\mbox{higher terms})$
then $k\geq 0$ is the least such that $\alpha_{r,n,k}\in\Z_p^*$ for all 
prime $p\equiv r\bmod d$ and $p>d+k$.
Conversely, 
if $k\geq 0$ is the least such that $\alpha_{r,n,k}\in\Z_p^*$ for all 
prime $p\equiv r\bmod d$ and $p>\max(d+k,\MaxPrime(h_{r,n,k}))$ then 
we have
$H_{r,n}=h_{r,n,k}X^k+(\mbox{higher terms})$.
\end{enumerate}
\end{proposition}

\begin{proof} 
Part (1) follows from the definition of $H_{r,n}$ in (\ref{E:H_n}).
Part (2) follows from Lemma \ref{L:basic}.
Part (3) follows from Lemma \ref{L:kappa}:
since for $p>d+k$ we have 
$\kappa_{r,n,k}\in\Z_p^*$ and 
by Lemma \ref{L:kappa}
$\alpha_{r,n,k}=\tilde{h}_{r,n,k}(\pfloor{\frac{p}{d}})/\kappa_{r,n,k}$, we have
that 
$\tilde{h}_{r,n,k}(\pfloor{\frac{p}{d}}) \in\Z_p^*$
if and only if $\alpha_{r,n,k}\in\Z_p^*$.
It is clear that Part (4) follows from (1)--(3).
\end{proof}

\section{Tame $A$-determinant}

This section completely determines 
the $p$-adic order of certain  finite {\em tame} $A$-determinant.
These tame $A$-determinants will be used to approximate 
our Fredholm $A$-determinant in Section 4.
They are the bridge connecting the generating polynomials 
to the actually $p$-adic Fredholm determinant 
in Dwork theory.

Let $E_p(-)$ be the $p$-adic Artin-Hasse exponential function (see \cite{Kob84}).
We pick a root $\gamma$ of $\sum_{i=1}^{\infty}\frac{x^{p^i}}{p^i}$ in $\bar{\Q}$ of 
$\ord_p \gamma = 1/(p-1)$ such that 
$\zeta_p=E_p(\gamma)$ is the same primitive $p$-th root of unity as
in the beginning and throughout of this paper.
For any integer $pi-j$ with $1\leq i,j\leq d-1$ 
we define and polynomial in $\Q[\gamma][A]$ for every $\ell^o\in\Z_{\geq 0}$
\begin{eqnarray}\label{E:F}
F_{pi-j,\ell^o}(A) &:=&  \sum_{\ell=0}^{\ell^o}\frac{A^{m_{ij}^\ell}\gamma^{m_{ij}^\ell+n_{ij}^{\ell}}}{m_{ij}^\ell! n_{ij}^\ell!}
\end{eqnarray}
Define the $n$-th tame $A$-determinant  
\begin{eqnarray}\label{E:P_n}
P_{n,\ell^o}(A)&:=&\det((F_{pi-j,\ell^o})_{1\leq i,j\leq n}).
\end{eqnarray}
It lies in $\Q[\gamma][A]$ and its key property is provided below in 
the lemma.
Notice that $\Z_p[\gamma]=\Z_p[\zeta_p]$ is the ring of integers
in $\Q_p(\gamma)=\Q_p(\zeta_p)$.

\begin{lemma}\label{L:monomial} 
Let $1\leq n\leq d-1$, let $\ell_b^o=\pfloor{b/d}$
where $b\in\Z_{\geq 0}$. 

\begin{enumerate}
\item 
Then $P_{n,\ell_b^o}(A)$ 
can be written as a polynomial in 
$\Q[A\gamma^{1-\frac{s}{d}}]$
whose coefficients are monomials in $A\gamma^{1-\frac{s}{d}}$. 
Furthermore, we have
\begin{eqnarray*}
P_{n,\ell_b^o}(A) &=& 
\gamma^{\frac{(p-1)n(n+1)}{2d}}
\sum_{k_{r,n}^o\le k< k_{r,n}^o+b}
\alpha_{r,n,k}(A\gamma^{1-\frac{s}{d}})^k\\
&&+\gamma^{\geq\frac{(p-1)n(n+1)}{2d}+(1-\frac{s}{d})(k_{r,n}^o+b)}R
\end{eqnarray*}
for some $R\in\Z_p[\gamma][A]$.

\item 
If $H_{r,n}=h_{r,n,k_{r,n}}X^{k_{r,n}}+(\mbox{higher terms})$
then
for all $p\equiv r\bmod d$ and $p>\max(d+k_{r,n},\MaxPrime(h_{r,n,k_{r,n}}))$ 
\begin{eqnarray*}
\ord_p (P_{n,\ell_b^o}(A)) &=& \frac{n(n+1)}{2d}+\frac{(1-\frac{s}{d})k_{r,n}}{p-1}.
\end{eqnarray*}
Conversely, if
$\ord_p (P_{n,\ell_b^o}(A))= \frac{n(n+1)}{2d}+\frac{(1-\frac{s}{d})k_{r,n}}{p-1}$
 for $p>d+k_{r,n}$ then 
$H_{r,n}=h_{r,n,k_{r,n}}X^{k_{r,n}}+(\mbox{higher terms})$.

\item 
Let $a\in\bar\Q^*$ and let $\bar{a}$ be its residue reduction over $p$.
Let $\hat{a}$ be the Teichm\"uller lifting of $\bar{a}$.
If $H_r=\sum_{n=1}^{d-1}h_{r,n,k}X_{r,n}^{k_{r,n}}+(\mbox{higher terms})$ 
then  for all prime $p\equiv r\bmod d$ and $p> \max(d+\max_n(k_{r,n}),\MaxPrime_n(h_{r,n,k_{r,n}}), \MaxPrime(a))$, we have  for all $n$
\begin{eqnarray*}
\ord_p (P_{n,\ell_b^o}(\hat{a})) &=& \frac{n(n+1)}{2d}+\frac{(1-\frac{s}{d})k_{r,n}}{p-1}.
\end{eqnarray*}
Conversely, if 
$\ord_p (P_{n,\ell_b^o}(\hat{a}))= \frac{n(n+1)}{2d}+\frac{(1-\frac{s}{d})k_{r,n}}{p-1}$
for all $n$ and all prime $p> \max(d+\max_n(k_{r,n}),\MaxPrime(a))$
then 
$H_r=\sum_{n=1}^{d-1}h_{r,n,k}X_{r,n}^{k_{r,n}}+(\mbox{higher terms})$.
\end{enumerate}
\end{lemma}

\begin{proof}
(1)
By the formal expansion of determinant and the above identity, we have
\begin{eqnarray*}
P_{n,\ell_b^o} (A)&=&
\sum_{\sigma\in S_n} \sgn(\sigma)\prod_{i=1}^{n} F_{pi-\sigma(i),\ell_b^o}\\
&=& 
\sum_{\sigma\in S_n} \sgn(\sigma)\prod_{i=1}^{n} 
\sum_{\ell=0}^{\ell_b^o}\frac{A^{m_{i,\sigma(i)}^\ell}\gamma^{m_{i,\sigma(i)}^\ell
+ n_{i,\sigma(i)}^\ell}} 
{m_{i,\sigma(i)}^\ell! n_{i,\sigma(i)}^\ell!}\\
&=&
\sum_{\sigma\in S_n}\sgn(\sigma) 
\sum_{0\leq \ell_{ij}\leq \ell_b^o}\frac{A^{\sum_{i=1}^{n} m_{i,\sigma(i)}^{\ell_{i,\sigma(i)}}}}{
\prod_{i=1}^{n} m_{i,\sigma(i)}^{\ell_{i,\sigma(i)}}! n_{i,\sigma(i)}^{\ell_{i,\sigma(i)}}!} 
\gamma^{\sum_{i=1}^{n}m_{i,\sigma(i)}^{\ell_{i,\sigma(i)}}+n_{i,\sigma(i)}^{\ell_{i,\sigma(i)}} }
\end{eqnarray*}
Notice that by Lemma \ref{L:Frobenius} for any $\ell_{i,\sigma(i)}$
$$
d\sum_{i=1}^{n} n_{i,\sigma(i)}^{\ell_{i,\sigma(i)}} + 
s \sum_{i=1}^{n} m_{i,\sigma(i)}^{\ell_{i,\sigma(i)}}= (p-1)n(n+1)/2.
$$
Write $k=\sum_{i=1}^{n}m_{i,\sigma(i)}^{\ell_{i,\sigma(i)}}$. 
Then
$$
\sum_{i=1}^{n}m_{i,\sigma(i)}^{\ell_{i,\sigma(i)}}+n_{i,\sigma(i)}^{\ell_{i,\sigma(i)}}
=\frac{(p-1)n(n+1)}{2d} + (1-\frac{s}{d})k.
$$
Then there are $w_k\in\Z_p^*$ such that
\begin{eqnarray*}
P_{n,\ell_b^o}(A)&=& 
\gamma^{\frac{(p-1)n(n+1)}{2d}} \sum_{k_{r,n}^o\leq k\leq k_{r,n}^o+b} 
\alpha_{r,n,k} A^k\gamma^{(1-\frac{s}{d})k}\\
&&+\gamma^{\frac{(p-1)n(n+1)}{2d}}\sum_{k\geq k_{r,n}^o+b}
\sum_{\sigma\in S_n^k}
\sgn(\sigma) w_kA^k\gamma^{(1-\frac{s}{d})k} 
\\ 
&=&\gamma^{\frac{(p-1)n(n+1)}{2d}} 
     \sum_{k_{r,n}^o\leq k\leq k_{r,n}^o+b} 
     \alpha_{r,n,k} (A\gamma^{(1-\frac{s}{d})})^k\\
&&+\gamma^{\geq \frac{(p-1)n(n+1)}{2d}+(1-\frac{s}{d})k_{r,n}} R
\end{eqnarray*}
for some $R\in\Z_p[\gamma][A]$. 

(2)
Fix $n$.
By Proposition  \ref{P:A^k}
our hypothesis implies $k_{r,n}$ is the least $k$ such that
$\alpha_{r,n,k_{r,n}}\in \Q\cap\Z_p^*$ 
for $p\equiv r\bmod d$ and $p>N_r$.
For all $k_{r,n}^o\leq k<k_{r,n}$ 
we have $\alpha_{r,n,k_{r,n}}\in \Q\cap p\Z_p$.
Hence $\ord_p \alpha_{r,n,k_{r,n}}\geq 1$.
Thus 
the $p$-adic valuations are precisely as displayed 
by our part (1).

(3) 
Fix $n$.  
Since $a\neq 0$ we have for 
$p>\MaxPrime(a)$ then  $\hat{a}\in\bar\Z^*$. 
Consider the formula in Part (1) 
\begin{eqnarray*}
P_{n,\ell_b^o}(\hat{a}) &=& 
\gamma^{\frac{(p-1)n(n+1)}{2d}}
\sum_{k_{r,n}^o\le k< k_{r,n}^o+b}
\alpha_{r,n,k}\hat{a}^k \gamma^{(1-\frac{s}{d})k}\\
&&+\mbox{(higher $\gamma$-terms)}
\end{eqnarray*}
Then applying an analogous argument of Part (2) 
we conclude that
$H_{r,n}$ has lowest degree term $h_{r,n,k_{r,n}}X^{k_{r,n}}$
if and only if $k_{r,n}$ is the least $k$ such that 
$\alpha_{r,n,k}\hat{a}^k\in\bar\Z_p^*$; and
hence it is equivalent to 
$$\ord_p P_{n,\ell_b^o}(\hat{a}) 
=\frac{n(n+1)}{2d}+\frac{(1-\frac{s}{d})k_{r,n}}{p-1}. 
$$
This proves our statements.
\end{proof}

\section{Asymptotic Dwork theory for $A$-families}

In this section we approximate Fredholm $A$-determinant 
by those tame determinants defined in Section 3.
To keep the paper short 
we refer the reader to
\cite{AS89}, \cite{Wan93} and \cite{Wan04}
for more thorough treatment of 
classical Dwork theorem.
Let $f(x)=x^d+ax^s$ be a polynomial with $a\in\bar\Q$ and $d>s\geq 1$ are coprime 
integers. Namely, $f(x)\in \A(s,d)(\bar\Q)$.
Let $\bar{a}$ be the reduction mod $\wp$ of $a$ for a prime ideal $\wp$ 
in the number field $\Q(a)$. Let $\hat{a}$ be the $p$-adic
Teichm\"uller lifting of $\bar{a}$ in $\bar\Z_p$.
We recall the Dwork trace formula for the $L$ function of exponential sum of 
$\bar{f}=f(x)\bmod \wp$, assuming $\wp$ has residue field $\F_q$ 
for some $p$-power $q$. 
Let $\zeta_p$ be the primitive $p$-th root of unity fixed since the first section of this paper.
Let $\gamma \in\bar\Q_p$ be the root of 
$\log_p E_p(x)=\sum_{i=0}^{\infty}\frac{x^{p^i}}{p^i}$
with $\ord_p(\gamma)=1/(p-1)$ such that 
$E_p(\gamma)=\zeta_p$.
Write $E_p(\gamma X) =  \sum_{t=0}^{\infty}\lambda_tX^t$
for some $\lambda_t\in (\Q\cap\Z_p)[\gamma]$. 
Then we have
$\lambda_t = \gamma^t/t!$ for all $0\leq t\leq p-1$, and  
$\ord_p \lambda_t \geq t/(p-1)$ for all $t\geq 0$.
For any integer $v\geq 0$ let  
\begin{eqnarray}
F'_v(A) &:=&\sum_{n_v,m_v}  \lambda_{n_v} \lambda_{m_v}A^{m_v}
\end{eqnarray}
where the sum ranges over $m_v,n_v\in\Z_{\geq 0}$ such that $n_vd+m_vs=v$.
For the only situation we are studying in this paper 
$v=pi-j$ with $1\leq i,j\leq d-1$ 
we use the notation from Lemma \ref{L:Frobenius}
that is,
$m_{pi-j}=m_{ij}^\ell$
and 
$n_{pi-j}=n_{ij}^\ell$.

From now on we assume $p>s(d-1)$ (so as to apply Lemma \ref{L:Frobenius}).
Recall from (\ref{E:F}) and for $\ell_b^o=\pfloor{b/d}$ with $b\in\Z_{\geq 0}$
\begin{eqnarray*}
F_{pi-j,\ell_b^o}(A) &=& 
\sum_{\ell=0}^{\ell_b^o}\frac{A^{m_{ij}^\ell}\gamma^{m_{ij}^\ell+n_{ij}^\ell}}{m_{ij}^\ell! n_{ij}^\ell!}.
\end{eqnarray*}
Then we have 
\begin{eqnarray*}
F'_{pi-j}(A) 
&=& 
\sum_{\ell\geq 0}
u_{i,j,\ell}A^{m_{ij}^\ell}\gamma^{m_{ij}^\ell+n_{ij}^\ell}\\
&=&
F_{pi-j,\ell_b^o}(A)+ \sum_{\ell>\ell_b^o}
u_{i,j,\ell}A^{m_{ij}^\ell}\gamma^{m_{ij}^\ell+n_{ij}^\ell}
\end{eqnarray*} 
for some $u_{i,j,\ell}\in\Z_p[\gamma]$ which is equal to 
$\frac{1}{m_{ij}^\ell!n_{ij}^\ell!}$ when $\ell\leq \ell_b^o$.
Let $P_{n,\ell_b^o}(A)=\det(F_{pi-j,\ell_b^o}(A))_{1\leq i,j\leq  n}$ for all $1\leq n\leq d-1$. We show below that $P_{n,\ell_b^o}(A)$ approximates 
$P'_n(A):=\det(F'_{pi-j})_{1\leq i,j\leq n}$ up to $b$ terms $p$-adically.

\begin{lemma}\label{L:approx}
Write the generating function 
$H_r=\sum_{k\geq k_{r,n}^o}\sum_{n=1}^{d-1}h_{r,n,k}X_{r,n}^k$.

(1) 
If $H_r=\sum_{n=1}^{d-1}h_{r,n,k_{r,n}}X_{r,n}^{k_{r,n}}+(\mbox{higher terms})$,
we write $N_r:=
\max(s(d-1), d+\max_n(k_{r,n}),\MaxPrime_n(h_{r,n,k_{r,n}}))$,
then for all $n$ for all prime $p\equiv r\bmod d$ with 
$p>N_r$ we have
$$
\ord_p(P'_n(A)) =\ord_p(P_{n,\ell_b^o}(A)) =
\frac{n(n+1)}{2d}+\frac{(1-\frac{s}{d})k_{r,n}}{p-1}.
$$
Conversely, if 
$
\ord_p(P'_n(A)) =\ord_p(P_{n,\ell_b^o}(A)) =
\frac{n(n+1)}{2d}+\frac{(1-\frac{s}{d})k_{r,n}}{p-1}
$
for all $n$ and all prime $p>\max(s(d-1),d+\max_n(k_{r,n})))$
then we have
$H_r=\sum_{n=1}^{d-1}h_{r,n,k_{r,n}}X_{r,n}^{k_{r,n}}+(\mbox{higher terms})$.

(2)
Let $a\in\bar\Q^*$ and let $\bar{a}$ be its residue reduction over $p$.
Let $\hat{a}$ be the Teichm\"uller lifting of $\bar{a}$.
If $H_r=\sum_{n=1}^{d-1} h_{r,n,k_{r,n}}X_{r,n}^{k_{r,n}}+(\mbox{higher terms})$
then for all prime $p\equiv r\bmod d$ and 
$p>\max(N_r,\MaxPrime(a))$
for all $1\leq n\leq d-1$ we have 
\begin{eqnarray*}
\ord_p (P'_n(\hat{a})) &=& \frac{n(n+1)}{2d}+\frac{(1-\frac{s}{d})k_{r,n}}{p-1}.
\end{eqnarray*}
Conversely, if
for all $n$ and all prime $p\equiv r\bmod d$ and 
$p>\max(s(d-1), d+\max_n(k_{r,n}),\MaxPrime(a))$ 
we have 
$
\ord_p (P'_n(\hat{a})) = \frac{n(n+1)}{2d}+\frac{(1-\frac{s}{d})k_{r,n}}{p-1}
$
then the generating function is of the form
$H_r=\sum_{n=1}^{d-1} h_{r,n,k_{r,n}}X_{r,n}^{k_{r,n}}+(\mbox{higher terms})$.
\end{lemma}

\begin{proof}
(1) 
Let $1\leq i,j\leq d-1$. Let $p>s(d-1)$.
Then we have for some $u_{i,\sigma(i),\ell}\in\Z_p$ that 
\begin{eqnarray*}
P'_n(A)
&=& 
\sum_{\sigma\in S_n} \sgn(\sigma) \prod_{i=1}^{n} \sum_{\ell=0}^{\infty}
u_{i,\sigma(i),\ell} A^{m_{i,\sigma(i)}^\ell}
\gamma^{m_{i,\sigma(i)}^\ell + n_{i,\sigma(i)}^\ell}.
\end{eqnarray*}
Using the same computational argument as that of Lemma \ref{L:monomial}
we get 
\begin{eqnarray*}
&=& 
\gamma^{\frac{(p-1)n(n+1)}{2d}}
\sum_{k_{r,n}^o\leq k< k_{r,n}^o+b}
\left(
\sum_{(\sigma,\ell_{i,\sigma(i)})\in\cS(k)}
\frac{\sgn(\sigma)A^k}{\prod_{i=1}^{n}m_{i,\sigma(i)}^{\ell_{i,\sigma(i)}} ! 
n_{i,\sigma(i)}^{\ell_{i,\sigma(i)}}!}\right) 
\gamma^{(1-\frac{s}{d})k}\\
&&+
\gamma^{\frac{(p-1)n(n+1)}{2d}}
\sum_{k\geq k_{r,n}^o+b}
\sum_{(\sigma,\ell_{i,\sigma(i)})\in\cS(k)}\sgn(\sigma)w_k A^k
\gamma^{(1-\frac{s}{d})k}
\end{eqnarray*}
for some $w_k\in\Z_p[\gamma]$.
By Lemma \ref{L:kappa} we can write 
\begin{eqnarray*}
P'_n(A) &=&
\gamma^{\frac{(p-1)n(n+1)}{2d}}\sum_{0\leq k-k_{r,n}^o<b}
\alpha_{r,n,k} A^k \gamma^{(1-\frac{s}{d})k}\\
&&+\gamma^{\geq\frac{(p-1)n(n+1)}{2d}+(1-\frac{s}{d})(k_{r,n}^o+b)}R
\end{eqnarray*}
for some $R\in\Z_p[\gamma][A]$.
Since $H_r\neq 0$ we have a minimal such $k$, 
denoted by $k_{r,n}$, such that 
$\tilde{h}_{r,n,k}(-\frac{r}{d})\neq 0$ and 
$0\leq k-k_{r,n}^o<b$. By Proposition \ref{P:A^k} for $p\equiv r\bmod d$ 
with $p>\max(d+\max_n(k_{r,n}),\MaxPrime_n(h_{r,n,k_{r,n}}))$ 
we have
$\alpha_{r,n,k}\in\Z_p^*$ and hence
$$
\ord_p(P'_n(A))=
\frac{n(n+1)}{2d}+\frac{(1-\frac{s}{d})k_{r,n}}{p-1}.
$$
Comparing with Lemma \ref{L:monomial}
$$
\ord_p (P_{n,\ell_b^o}(A)) = \frac{n(n+1)}{2d}+\frac{(1-\frac{s}{d})k_{r,n}}{p-1}
$$

The converse direction follows by applying Proposition \ref{P:A^k} 
again with analogous argument as that of Lemma \ref{L:monomial}.

(2) 
The proof here is analog to that of Lemma \ref{L:monomial}
by applying Proposition \ref{P:A^k} by applying the extra condition
that $p>s(d-1)$ on top of both direction.
\end{proof}

Then we prove Theorem \ref{T:main} below by applying 
the $p$-adic Dwork theory and transformation theorem we 
developed in \cite{Zhu12}.

\begin{theorem}[Theorem \ref{T:main}]
\label{T:Dwork}
Suppose $H_r\in\Q[X_{r,1},\ldots,X_{r,d-1}]$ 
defined in (\ref{E:H}) is nonzero with lowest degree terms 
 $\sum_{n=1}^{d-1}h_{r,n,k_{r,n}}X_{r,n}^{k_{r,n}}$.
Let $N_{s,d,r}$ be defined as in Theorem \ref{T:main}.
Then for $p\equiv r\bmod d$ and $p>N_{s,d,r}$ we have
$\GNP(\A(s,d),\bar\F_p)$ with
breaking points after origin 
$$
(n,\frac{n(n+1)}{2d}+\frac{(1-\frac{s}{d})k_{r,n}}{p-1})
$$
for $n=1,\ldots,d-1$.

Conversely, suppose for all prime $p\equiv r\bmod d$ and 
$p>\max(s(d-1), d+\max_n(k_{r,n}),2(d-s)\max_n(k_{r,n}))$
we have  
$\GNP(\A(s,d),\bar\F_p)$ is of the above form (it necessarily breaks at each point), then 
$H_r=\sum_{n=1}^{d-1}h_{r,n,k_{r,n}}X_{r,n}^{k_{r,n}}+(\mbox{higher terms})$.

Given $f=x^d+ax^s\in \A(s,d)(\bar\Q)$ 
with $\bar{f}=x^d+\bar{a}x^s \in\A(s,d)(\F_q)$. 
If $a\in\bar\Q^*$, then for all prime $p\equiv r\bmod d$ 
and $p>\max(N_{s,d,r},\MaxPrime(a))$ we have
$$
\NP(\bar{f}) =\GNP(\A(s,d),\bar\F_p).
$$
Furthermore, we have
$$
\lim_{\stackrel{p\rightarrow\infty}{p\equiv r\bmod d}}
\NP(\bar{f}) =\HP(\A(s,d)).
$$
\end{theorem}

\begin{proof}
Let $a\neq 0$.
We define a twisted matrix $M'':=(F''_{pi-j}):=(F'_{pi-j}\gamma^{j-i})$,
notice this is the matrix representing the Dwork operator 
with respect to a weighted monomial basis.
For $q=p^c$ for 
write 
$$(M''/\F_q)(A)
:= M''\cdot {M''}^{\tau^{-1}} \cdot {M''}^{\tau^{-2}}\cdots 
\cdot {M''}^{\tau^{-(c-1)}}
$$
where $\tau$ is the Frobenius map on $\Q_q(\zeta_p)$
that fixes $\Q_p(\zeta_p)$
that lifts the Frobenius map $x\mapsto x^p$ over its residue field
extension, and $\tau(A)=A^p$.
Then Dwork theory states that
\begin{eqnarray}
L(\bar{f}/\F_q,T) = \frac{\det(1-T(M''/\F_q)(\hat{a}))}{\det(1-qT(M''/\F_q)(\hat{a}))}
\end{eqnarray}
and it is of the form $1+C_1T+\cdots +C_{d-1}T^{d-1}$
in $\Z[\zeta_p][T]$.

Since 
$$
\ord_p F'_{pi-j}(\hat{a})\geq 
\frac{\pceiling{\frac{pi-j}{d}}}{p-1}\geq \frac{i}{d} + \frac{i-j}{d(p-1)},
$$
we have
$\ord_p (F''_{pi-j})\geq \frac{i}{d}$ for every $i,j\geq 1$.
Write $P''_n:=\det((M'')^{[n]})$. Obviously 
$P''_n(\hat{a}) =P'_n(\hat{a})$.
Apply Lemma \ref{L:approx}, 
we have that for $p\equiv r\bmod d$ and $p>\max(N_r,\MaxPrime(a))$
and for all $1\leq n\leq d-1$
$$
\ord_p(P'_n(\hat{a}))=\ord_p(P_{n,\ell_b^o}(\hat{a}))
=\frac{n(n+1)}{2d}+\frac{(1-\frac{s}{d})k_{r,n}}{p-1}.
$$
In summary, we have
$$\ord_p P''_n(\hat{a})=\ord_p P'_n(\hat{a})
=\frac{n(n+1)}{2d}+\frac{(1-\frac{s}{d})k_{r,n}}{p-1}.
$$
Thus for $p>2(d-s)k_{r,n}+1$ we have
$$\sum_{i=1}^{n}\frac{i}{d}=\frac{n(n+1)}{2d}\leq 
\ord_pP''_n(\hat{a})
< \frac{n(n+1)+1}{2d}.
$$
This verifies that the hypothesis of the tranform theorem 
in Section 5 of \cite{Zhu12} is satisfied, hence we are 
enabled to conclude that
$$
\NP(\bar{f})
=\NP_p(\sum_{n=0}^{d-1}P''_n(\hat{a}) T^n)
$$
and its breaking points after the origin are given by 
$$(n,\ord_p P''_n(\hat{a}))
=(n,\frac{n(n+1)}{2d}+\frac{(1-\frac{s}{d})k_{r,n}}{p-1})$$
for $n=1,\ldots,d-1$.

Conversely, suppose we know for such prime $p\equiv r\bmod d$ 
the breaking points of $\GNP(\A(s,d),\bar\F_p)$ are as given.
Then we may apply the transform lemma of \cite{Zhu12} 
and conclude that it is equal to 
$\NP_p(\sum_{n=0}^{d-1}P_n'(\hat{a})T^n)$, or 
in other words for all $1\leq n\leq d-1$ we have
$$\ord_p P'_n(\hat{a}) = \frac{n(n+1)}{2d}+\frac{(1-\frac{s}{d})k_{r,n}}{p-1}.$$
Then we apply Lemma \ref{L:approx} and find 
that $H_r$ has its lowest 
degree terms in the given form.

The last statement follows by taking limit.
\end{proof}

\begin{corollary}
\label{C:main2}
Let notation be as in Theorem \ref{T:Dwork}.
Suppose $H_r$ is nonzero with lowest degree terms
of the form $\sum_{n=1}^{d-1}h_{r,n,k}X_{r,n}^k$ for every $2\leq r\leq d-1$. 
Let $f=x^d+ax^s\in\A(s,d)(\bar\Q)$ with $d>s\geq 1$ coprime.
Then  for all prime $p>\max_r(N_{s,d,r},\MaxPrime(a))$ we have that 
\begin{eqnarray}
\NP(\bar{f}) &=&\GNP(\A(s,d);\bar\F_p)
\end{eqnarray}
and
$\lim_{p\rightarrow\infty}\NP(\bar{f}) =\HP(\A(s,d))$
if and only if $a\neq 0$.
\end{corollary}

\begin{proof}
Suppose $a\neq 0$ then the statement follows from 
Theorem \ref{T:Dwork}. If $a=0$ then $f=x^d$ 
and $\NP(\bar{f})$ is explicitly worked out by Stickelberger theorem (see {Wan04}). 
For $p\equiv 1\bmod d$ we have $\NP(\bar{f})=\HP(\A(s,d))$
but for $2\leq r\leq d-1$ we know  
$\NP(\bar{f})$ lies strictly above $\GNP(\A(s,d),\bar\F_p)$.
hence $\lim_{p\rightarrow\infty}\NP(\bar{f})$ does not exist.
\end{proof}

For any $s<d$ coprime integers and for any 
$q=p^c$ ($c\in\Z_{\ge 1}$), define 
$$
\GNP(\A(s,d),\F_q)
:=
\inf_{\bar{f}\in\A(s,d)(\F_q)}\NP(\bar{f}) 
$$
if exists.
Grothendieck-Katz specialization theorem implies that $\GNP(\A(s,d),\bar\F_p)$ exists.
Our proof of the main theorem implies the following statement immediately.

\begin{corollary}
Let notation be as in Theorem \ref{T:main}. For $p$ large enough, 
$\GNP(\A(s,d),\F_q)$ exists for any $p$-power $q$ and we have
$$
\GNP(\A(s,d),\F_q)=\GNP(\A(s,d),\bar\F_p).
$$
\end{corollary}

\begin{remark}\label{R:4.4}
The computation of $H_r$ starts with smallest $k\geq k_{r,n}^o$
and increases until we find the next term with $h_{r,n,k}\neq 0$.
When $s=1$ we have $H_r\neq 0$ 
with lowest degree term  
$\sum_{n=1}^{d-1}h_{r,n,k}X_{r,n}^{k_{r,n}^o}$ 
(it is shown in \cite{Zhu03}), namely it achieves its lowest possible degree.
But for $s\geq 2$ it is not always true that $H_r\neq 0$ with 
lowest degree term equal to $\sum_{n=1}^{d-1}h_{r,n,k}X_{r,n}^{k_{r,n}^o}$.
In fact in the case $(s,d)=(2,5)$ and $r=3$ one 
can show directly that $H_3$ has its least degree monomial of strictly higher degree
then $k_{r,n}^o$ for at least one $n$.
\end{remark}

\end{document}